\documentclass[12pt,a4paper]{article}
\usepackage{amssymb}
\usepackage[english]{babel}
\usepackage{amsmath}
\usepackage{amsthm}
\usepackage{amsfonts}
\usepackage{graphicx}
\usepackage{calrsfs}
\usepackage[latin1]{inputenc}

\newtheorem{theorem}{Theorem}
\newtheorem{lema}{Lemma}
\newtheorem{prop}{Proposition}
\newtheorem{cor}{Corollary}

\newtheorem{remark}{Remark}

\newcommand{\R}{\mathbb{R}}

\newcommand{\sfe}{{\mathbb S}^{n-1}}
\newcommand{\ms}{{\mathcal H}^{n-1}}

\def\p{\partial}

\def\det{\mathop\mathrm{det}}
\def\N{\mathbb{N}}
\def\sph{\mathbb{S}^{n-1}}

\def\beq{\begin{equation}}
\def\eeq{\end{equation}}

\begin{document}

\title{Functional inequalities derived from the Brunn--Minkowski
  inequalities for quermassintegrals}
\author{Andrea Colesanti \& Eugenia Saor\'{\i}n G\'{o}mez\footnote{Supported
  by EU Project {\em Phenomena in High Dimensions} MRTN-CT-2004-511953.}}
\date{}

\maketitle
\begin{abstract}
\noindent
We use Brunn--Minkowski inequalities for quermassintegrals
to deduce a family of inequalities of Poincar\'e type on the unit sphere and
on the boundary of smooth convex bodies in the $n$--dimensional Euclidean
space.
\end{abstract}

\bigskip


\section{Introduction}

The main idea of this paper is to use Brunn--Minkowski inequalities
for quermassintegrals to derive a family of inequalities of
Poincar\'e type on the unit sphere and on the boundary of convex
bodies in the $n$--dimensional Euclidean space. This type of
research was initiated in \cite{Colesanti} where the case of the
classic Brunn--Minkowski inequality is considered.

Let $K\subset\R^n$ be a convex body, i.e. a (non--empty) compact
convex set. The quermassintegrals of $K$, denoted by $W_0(K)$,
$W_1(K)$, $\dots,W_n(K)$, arise naturally in the polynomial
expression of the volume of the outer parallel bodies of $K$ given
by the well known {\it Steiner formula}:
$$
{\cal H}^n(K+tB)=\sum_{i=0}^n\,t^i\binom ni\,W_i(K)\,,\quad t\ge0\,.
$$
where $B$ is the unit ball of $\R^n$, $K+tB=\{x+ty\,:\,x\in
K\,,\,y\in B\}$ is the outer parallel body of $K$ at distance
$t\geq 0$ and ${\mathcal H}^n$ is the $n$--dimensional Lebesgue
measure. For a detailed study of quermassintegrals
we refer to \cite[\S 4.2]{Schneider}. Some of the
quermassintegrals have familiar geometric meaning: $W_0(K)$ is the
volume (i.e. the Lebesgue measure) of $K$, while $W_1(K)$ is, up
to a dimensional factor, the surface area of $K$. Each
quermassintegral $W_i$, $i<n$, satisfies a Brunn--Minkowski type
inequality: for every $K$ and $L$ convex bodies and for every
$t\in[0,1]$ we have
\begin{equation}\label{0.01}
W_i((1-t)K+tL)^{1/(n-i)}\ge
(1-t)W_i(K)^{1/(n-i)}+tW_i(L)^{1/(n-i)}\,;
\end{equation}
for $i<n-1$ equality holds if and only if $K$ is homothetic to $L$.
When $i=0$ this is the classic Brunn--Minkowski inequality. In
general, the above inequalities can be obtained as consequences of
the Aleksandrov--Fenchel inequalities (see for instance \cite[\S
6.4]{Schneider}). Inequality (\ref{0.01}) claims that the
functional $W_i^{1/(n-i)}$ is concave in the class of convex
bodies; heuristically, this implies that the {\it second
variation} of this functional, whenever it exists, must be
negative semi--definite. In this paper we try to make this argument
more precise and we study its consequences.

Throughout the paper we use
the notion of elementary symmetric functions of (the eigenvalues of)
symmetric matrices. In our notation, if $A$ is a $N\times N$
real symmetric matrix, for $r\in\{0,1,\dots,N\}$, $S_r(A)$ is the $r$--th
elementary symmetric function of the
eigenvalues of $A$ and $(S^{ij}_r(A))$ is the {\it $r$--cofactor matrix} of $A$;
these notions and their properties are recalled in \S \ref{II}.

If $K\subset\R^n$ is a convex body of class $C^2_+$ (see \S \ref{II} for the definition)
then, for $i<n$,
\begin{equation}
%\label{I.0}
W_i(K)=c(n,i)\int_{\sfe} h_K\,S_{n-i-1}((h_K)_{ij}+h_K \delta_{ij})\,d\ms\,,
\end{equation}
where $c(n,i)$ is a constant and $(h_K)_{ij}$ are the second covariant
derivatives of the support function $h_K$ of $K$ (see formula
(5.3.11) in \cite{Schneider} for the value of $c(n,i)$ and \S 2 for precise
definitions). This integral representation formula allows to
compute explicitly the first and second directional derivatives of
quermassintegrals. Then, imposing the Brunn--Minkowski inequality
(\ref{0.01}) we obtain the following results.

\begin{theorem}\label{teo0.1}
Let $K\subset\R^n$ be a convex body of class $C^2_+$, $\nu$ be its Gauss map and $I\in\{1,\dots,n-1\}$. For
every $\psi\in C^1(\p K)$, if
\begin{equation}
\label{0.02}
\int_{\p K}{\psi S_{I-1}(D\nu)d\mathcal{H}^{n-1}}=0
\end{equation}
then
\begin{equation}\label{0.03}
I\,\int_{\p K}\psi^2{S_{I}(D\nu)d\mathcal{H}^{n-1}}\leq \int_{\p
K}{\left<(S_{I}^{ij}(D\nu))\nabla \psi,(D\nu)^{-1}\nabla \psi
\right>}d\mathcal{H}^{n-1}\,.
\end{equation}
\end{theorem}

\begin{theorem}\label{teo0.2}
Let $h$ be the support function of a convex body $K\subset\R^n$ of class
$C^2_+$ and $J\in\{1,\dots,n-1\}$. For every $\phi\in C^1(\sfe)$, if
\beq\label{0.1}
\int_{\sph}{\phi S_{J}(h_{ij}+h\delta_{ij})
d\mathcal{H}^{n-1}}=0
\eeq
then
\begin{equation}\label{0.1b}
(n-J)\int_{\sph}{\phi^2 S_{J-1}(h_{ij}+h\delta_{ij})}d\mathcal{H}^{n-1}\leq
\int_{\sph}\left<(S_{J}^{ij}(h_{ij}+h\delta_{ij}))\nabla\phi,\nabla\phi\right>d\mathcal{H}^{n-1}\,.
\end{equation}
\end{theorem}

Theorems \ref{teo0.1} and \ref{teo0.2} are the two faces of the same coin; they can
be obtained one from each other by the change of variable provided by the
Gauss map. The cases $I=1$ of Theorem \ref{teo0.1} and $J=n-1$ of
Theorem \ref{teo0.2} were already proved in \cite{Colesanti}, as consequences
of the classic Brunn--Minkowski inequality. Another proof of Theorems
\ref{teo0.1} and \ref{teo0.2} in these special cases, based on a functional inequality due to
Brascamp and Lieb (see \cite{Brascamp-Lieb}), was communicated to us by
Cordero--Erausquin (\cite{Cordero}).

One way to look at (\ref{0.02})--(\ref{0.03}) and (\ref{0.1})--(\ref{0.1b}) is as inequalities
of Poincar\'e type, where a weighted $L^2$--norm of a function is
bounded by a weighted $L^2$--norm of its gradient, under a zero--mean type
condition. In particular, choosing $K=B$ (the unit ball) in Theorem
\ref{teo0.1}, or equivalently $h\equiv1$ in Theorem \ref{teo0.2}, we recover
the usual Poincar\'e inequality on $\sfe$ with the optimal constant:
\begin{equation}
\label{0.1c}
\int_{\sfe}\phi(x)\, d\mathcal{H}^{n-1}(x)=0\;\Rightarrow\;
\int_{\sfe}\phi^2(x)d\mathcal{H}^{n-1}(x)\le
\frac{1}{n-1}\int_{\sfe}|\nabla\phi(x)|^2d\mathcal{H}^{n-1}(x)\,.
\end{equation}

We also note that inequalities (\ref{0.03}) and (\ref{0.1b}), under side conditions
(\ref{0.02}) and (\ref{0.1}) respectively, are optimal. This fact, proved in
Remark \ref{ossIII.1}, \S \ref{IV}, is a simple consequence of the invariance of
quermassintegrals under translations.

When $J=1$ we can remove the smoothness assumption on $K$ (or
equivalently on $h$) in Theorem \ref{teo0.2}. Indeed we have
$S_{J-1}=S_0\equiv1$ and $S_{1}^{ij}(h_{ij}+h\delta_{ij})=\delta_{ij}$.
Moreover $S_{1}(h_{ij}+h\delta_{ij})
d\mathcal{H}^{n-1}=[\Delta\,h+(n-1)h]d\mathcal{H}^{n-1}$ can be replaced by
$dA_1(K,\cdot)$, where $A_1(K,\cdot)$
denotes the {\it area measure of order one} of $K$ (see \S \ref{IV} for the definition).

\begin{theorem}\label{teo0.3} Let $K\subset\R^n$ be a convex body and let
  $A_1(K,\cdot)$ be its area measure of order one. For every $\phi\in
  C^1(\sfe)$, if
\begin{equation}\label{0.2}
\int_{\sfe}\phi(x)\, dA_1(K,x)=0\,,
\end{equation}
then
\begin{equation*}
%\label{0.3}
\int_{\sfe}\phi^2(x)d\mathcal{H}^{n-1}(x)\le
\frac{1}{n-1}\int_{\sfe}|\nabla\phi(x)|^2d\mathcal{H}^{n-1}(x)\,.
\end{equation*}
\end{theorem}

Hence Theorem \ref{teo0.3} extends the usual Poincar\'e
inequality (\ref{0.1c}) on $\sfe$ when the zero--mean condition is
replaced by (\ref{0.2}). For $n=2$ this leads to an extension
of the well known {\it Wirtinger
inequality}, stated in Corollary \ref{cor dim2} of \S \ref{V}. In higher
dimension Theorem \ref{teo0.3} together with some recent developments on the
Christoffel problem (\cite{Guan-Ma}, \cite{Sheng-Trudinger-Wang}) leads to the
following result.

\begin{theorem}\label{teo0.4}
Let $K\subset\R^n$ be a convex body containing the origin in its interior, such that
\begin{equation}\label{0.4}
\int_{\sfe}x\rho_K(x)\,d\mathcal{H}^{n-1}(x)=0\,,
\end{equation}
where $\rho_K$ is the radial function of $K$. Then, for every $\phi\in C^1(\mathbb{S}^{n-1})$,
$$
\int_{\sfe}\phi(x)\rho_K(x)\,d\mathcal{H}^{n-1}(x)=0
\;\Rightarrow\;
\int_{\sfe}\phi^2(x)d\mathcal{H}^{n-1}(x)\le
\frac{1}{n-1}\int_{\sfe}|\nabla\phi(x)|^2d\mathcal{H}^{n-1}(x)\,.
$$
\end{theorem}
Note that condition (\ref{0.4}) is fulfilled when $K$ is centrally symmetric.

\bigskip

\noindent{\bf Acknowledgment.} We would like to thank L. Al\'{\i}as Linares
for his precious help in the proof of Lemma \ref{lemmaII.1}.

\section{Preliminaries}\label{II}

\subsection{Elementary symmetric functions}
Let $N$ be an integer; for a $N\times N$ symmetric matrix
$A=(a_{ij})$ having eigenvalues $\lambda_1,\dots,\lambda_N$, and
for $k\in\{0,1,\dots,N\}$ we define
the $k$--th elementary symmetric function of the
eigenvalues of $A$ as follows
\beq
S_k(A)=
\sum_{1\leq i_1<\dots<i_k\leq
N}{\lambda_{i_1}\cdots\lambda_{i_k}}\,,\quad{\rm if}\;k\ge1\,,
\eeq
and $S_0(A)=1$. In particular $S_1(A)$ and $S_{N}(A)$
are the trace and the determinant of $A$, respectively.
If $A$ and $k$ are as
above and $i,j\in\{1,\dots,N\}$, we set
\begin{equation*}
S^{ij}_k(A)=\frac{\partial S_k(A)}{\partial a_{ij}}\,.
\end{equation*}
The matrix $(S^{ij}_k(A))$ is also symmetric.  The usual
cofactor matrix happens when $k=N$ in $(S^{ij}_k(A))$, so
$(S^{ij}_k(A))$ can be considered as a {\it $k$--th cofactor matrix
of $A$}. Note that $(S^{ij}_1(A))$ is the identity matrix. In the sequel we
will use some properties of elementary symmetric functions of
matrices that, for convenience, we gather in the following
statement; for the proof we refer the reader to \cite{Reilly} and
\cite[Chapter 1]{Salani}.

\begin{prop}\label{propII.1} In the notation introduced above the following
  facts hold
%
%Let $A=(a_{ij})$ be a $N\times N$, symmetric and positive
%  definite matrix, let $k\in\{1,\dots,N\}$ and let $(S^{ij}_k(A))$ be defined
%  as above. Then the following facts hold:
\begin{itemize}
\item[i)] $A$ is diagonal if and only if $(S^{ij}_k(A))$ is diagonal;
\item[ii)] the eigenvalues of $(S^{ij}_k(A))$ are given by
$$
\Lambda_s=S_{k-1}({\rm
diag}(\lambda_1,\dots,\lambda_{s-1},\lambda_{s+1},\dots,\lambda_{N}))\,,\quad
s=1,\dots,N\,,
$$
where $\lambda_1,\dots,\lambda_{N}$ are the eigenvalues of $A$;
\item[iii)] if $A$ is non--singular then
$$
\frac{1}{{\rm det}(A)}S_{k}(A)=S_{N-k}(A^{-1})\,;
$$
\item[iv)]
\begin{equation}
\label{I.1}
S_k(A)=\frac{1}{k}\sum_{i,j=1}^N S^{ij}_k(A)a_{ij}\,;
\end{equation}
\item[v)]
\begin{equation}\label{I.2}
{\rm trace}(S^{ij}_k(A))=(N-(k-1))S_{k-1}(A)\,.
\end{equation}
\end{itemize}
\end{prop}

\subsection{Convex bodies and quermassintegrals}
We denote by ${\mathcal K}^n$
the set of convex bodies in $\R^n$. In this paper we will use several results
concerning convex bodies, for the proof of these results we refer the reader to
\cite{Schneider}. To every $K\in{\mathcal K}^n$ we can associate
its {\it support function} $h_K$
$$
h_K\,:\,\mathbb{S}^{n-1}\to\R\,,\quad
h_K(u)=\sup\bigl\{\langle x,u\rangle:x\in K\bigr\},
$$
(see e.g. \cite[\S 1.7]{Schneider}). Note that in the present
paper the support function is defined on the unit sphere
$\mathbb{S}^n$ and we do not consider its homogeneous extension to
the whole space $\R^n$. $K$ is said to be of class $C^2_+$ if
$\partial K\in C^2$ and the Gauss curvature is strictly positive
at each point of $\partial K$. If $K$ is of class $C^2_+$ we
denote by $\nu_K$ its Gauss map: for every $x\in \partial K$,
$\nu_K(x)$ is the outer unit normal vector to $K$ at $x$. When the
body $K$ is clear from the context, we just write $h$ and $\nu$
instead of $h_K$ and $\nu_K$ respectively. If $K$ is of class $C^2_+$, then
$\nu_K$
establishes a diffeomorphism between $\partial K$
and $\sfe$ and its differential $D\nu_K$ is the {\it Weingarten
map} of $\partial K$. The matrix associated with the linear map
$D(\nu^{-1}) $ is $(h_{ij}+h \delta_{ij})$ where for
$i,j=1,\dots,n-1$, $h_i$ and $h_{ij}$ denote respectively the
first and second covariant derivatives of $h$ with respect to an
orthonormal frame on $\sfe$ and $\delta_{ij}$ is the standard
Kronecker symbol. In other words $(h_{ij}+h \delta_{ij})$ is the
matrix of the reverse second fundamental form of $\partial K$. For
brevity, in the sequel we will adopt the notation:
$$
(h_{ij}+h \delta_{ij})=\Xi^{-1}\,.
$$
In particular, if $K$ is of class $C^2_+$ then $\Xi^{-1}$ is positive definite
on $\mathbb{S}^{n-1}$ and its eigenvalues are the principal radii of curvature
of $K$. Conversely, if  $h\in C^2(\sfe)$ and
the matrix $(h_{ij}+h \delta_{ij})$ is positive definite at each point of
$\sfe$, then $h$ is the support function of a (uniquely determined) convex
body $K$ of class $C^2_+$. Hence the set
\[
\mathcal{C}=\{h\in C^2(\sph):(h_{ij}+h\delta_{ij})>0\;\mbox{on $\sfe$}\}
\]
consists of support functions of convex bodies of class $C^2_+$.

When $K$ is of class $C^2_+$, the quermassintegrals of $K$ can be
expressed as integrals involving the support function $h_K$ of
$K$. In fact, for $i\in\{0,1,\dots,n-1\}$,
\begin{equation}\label{I.0}
W_i(K)=\frac{1}{n}\,{\binom{n-1}{n-i-1}}^{-1}\int_{\sfe}
h_K\,S_{n-i-1}(\Xi^{-1})\,d\ms
\end{equation}
(see formula (5.3.11) in \cite{Schneider}). Note that for
$K,L\in{\mathcal K}^n$ and $t\in[0,1]$ we have
$$
h_{(1-t)K+tL}=(1-t)h_K+th_L\,.
$$
From the above facts and inequality (\ref{0.01}) we deduce the
following result.

\begin{prop}\label{BM for quermass}For $i\in\{0,1,\dots,n-1\}$
define the functional
$$
F_i\,:\,{\mathcal C}\to\R_+\,,\quad F_i(h)=\int_{\sfe}
h\,S_{n-i-1}(\Xi^{-1})\,d\ms\,.
$$
Then $(F_i)^{1/(n-i)}$ is concave in $\mathcal{C}$.
\end{prop}

\section{A lemma concerning Hessian operators on the sphere}
This section is devoted to prove the following result, which will be used in
the proofs of Theorems \ref{teo0.1} and \ref{teo0.2}.

\begin{lema}\label{lemmaII.1}
Let $u\in C^2(\sph)$, $k\in\{1,\dots,n-1\}$ and let $\{E_1,\dots
E_{n-1}\}$ be a local orthonormal frame of vector fields on $\sph$. Then, for
every $i\in\{1,\dots,n-1\}$,
\begin{equation*}
{\rm div}_j(S_k^{ij}(\nabla^2 u+uI)):=
\sum_{j=1}^{n-1}{\frac{\p}{\p E_j}S^{ij}_k(\nabla^2 u+uI)}=0\,,
\end{equation*}
where $\dfrac{\p}{\p E_j}$ denotes the covariant differential acting on $E_j$
and $I$ denotes the $(n-1)\times(n-1)$ identity matrix.
\end{lema}
The case $k=n-1$ of the preceding lemma was proved by Cheng and Yau in \cite{Cheng-Yau}
(see page 504). We also note that
an analogous result is valid in the Euclidean setting, with $(\nabla^2 u+uI)$
replaced by $\nabla^2 u$ (see for instance
\cite[Proposition 2.1]{Reilly} and \cite[\S 2.3]{Salani}). Our proof follows
the argument of \cite{Salani} for the Euclidean case and uses some standard tools from
differential geometry on $\sfe$.

\begin{proof} For $k\in\{0,1,\dots,N\}$, the $k$--th elementary symmetric
  functions of a symmetric $N\times N$ matrix $A=(a_{ij})$ can be written in
  the following way (see, for instance, \cite{Reilly})
\begin{equation*}
S_k(A)=\frac{1}{k}\sum{\delta\binom{i_1,\dots,i_k}{j_1,\dots,j_k}a_{i_1
j_1}\cdots a_{i_k j_k}}
\end{equation*}
where the sum is taken over all possible indices $i_r, j_r \in \{1,\dots,N\}$
  for $r=1,\dots,k$ and the Kronecker symbol
$\delta\binom{i_1,\dots,i_k}{j_1,\dots,j_k}$ equals $1$ (respectively, $-1$)
when $i_1,\dots,i_k$ are distinct and $(j_1,\dots,j_k)$ is an even
(respectively, odd) permutation of $(i_1,\dots,i_k)$; otherwise it is $0$.
Using the above equality we have
\begin{equation*}
S^{ij}_k(A)=\frac{1}{(k-1)!}\sum{\delta\binom{i,i_1,\dots,i_{k-1}}{j,j_1,\dots,j_{k-1}}a_{i_1
j_1}\cdots a_{i_{k-1} j_{k-1}}}\,.
\end{equation*}
Hence we can write
\begin{eqnarray}
\label{II.1}
&&(k-1)!\sum_{j=1}^{n-1}{\frac{\p}{\p E_j}S_k^{ij}(\nabla^2 u+uI)}=\\
&=&\sum_{j=1}^{n-1}\sum\delta\binom{i,i_1,\dots,i_{k-1}}{j,j_1,\dots,j_{k-1}}\frac{\p}{\p
E_j}
((u_{i_1 j_1}+u\delta_{i_1j_1})\cdots
(u_{i_{k-1}j_{\,k-1}}+u\delta_{i_{k-1}j_{\,k-1}}))\nonumber\\
&=&\sum_{j=1}^{n-1}\sum\delta\binom{i,i_1,\dots,i_{k-1}}{j,j_1,\dots,j_{\,k-1}}
[(u_{i_1 j_1 j}+u_j\delta_{i_1j_1})(u_{i_2j_2}+u\delta_{i_2j_2})\cdots
(u_{j_{k-1}i_{k-1}}+u\delta_{i_{k-1}j_{k-1}})+\nonumber\\
&&\cdots+(u_{i_1j_1}+u\delta_{i_1j_1})(u_{i_2j_2}+u\delta_{i_2j_2})\cdots
(u_{i_{k-1}j_{k-1}j}+u_j\delta_{i_{k-1}j_{k-1}})].\nonumber
\end{eqnarray}
In the last sum, for fixed $i_1,\dots i_{k-1},j_1,\dots j_{k-1},j$, let us
consider the terms
$$
A=\delta_1(u_{i_1j_1j}+u_j\delta_{i_1j_1})C\quad\mbox{and}\quad
B=\delta_2(u_{i_1jj_1}+u_{j_1}\delta_{i_1j_1})C\,,\;
$$ 
where
$$
\delta_1=\delta\binom{i,i_1,i_2,\dots,i_{k-1}}{j,j_1,j_2,\dots,j_{\,k-1}}\,,\;
\delta_2=\delta\binom{i,i_1,i_2,\dots,i_{k-1}}{j_1,j,j_2,\dots,j_{\,k-1}}\,,\;
$$
and
$$
C=(u_{i_2j_2}+u\delta_{i_2j_2})\cdots
(u_{j_{k-1}i_{k-1}}+u\delta_{i_{k-1}j_{k-1}})\,.
$$ 
Clearly $\delta_1=-\delta_2$. Moreover we have the following relation
concerning covariant derivatives on $\sfe$ (see, for instance, \cite{Cheng-Yau})
$$
u_{rst}+u_t\delta_{rs}=u_{rts}+u_s\delta_{rt}\,,\quad\forall\,
r,s,t=1,\cdots,n-1\,.
$$
Hence $A+B=0$. We have proved that to the term $A$ in the last sum in
(\ref{II.1}) it corresponds another term $B$, uniquely determined, which cancels out
with $A$. The same argument can be repeated for any other term of the sum and
this concludes the proof.
\end{proof}

\section{Proof of Theorems \ref{teo0.1} and \ref{teo0.2}}\label{IV}
In this section $K$ is a fixed convex body of class $C^2_+$ and $h$ is its
support function; in particular $h\in{\cal C}$. We recall that
$\Xi^{-1}=(h_{ij}+h\delta_{ij})$ and, for $k\in\{0,\dots,n-1\}$,
$$
F_k(h)=\int_{\sfe}h\,S_{n-k-1}(\Xi^{-1})\,d\ms\,.
$$
Note that if $\phi\in C^{\infty}(\sph)$ and $\epsilon$ is
sufficiently small, then $h+s\phi\in\mathcal{C}$ for
$|s|\le\epsilon$. We will denote by
$\Xi^{-1}_s$ the matrix $((h_s)_{ij}+h_s\delta_{ij})$.

\begin{prop}\label{propIV.1}
Let $k\in\{0,\dots,n-1\}$, $h\in \mathcal{C}$, $\phi\in C^{\infty}(\sph)$ and
$\epsilon>0$ be such that $h_s=h+s\phi\in \mathcal{C}$ for every
$s\in(-\epsilon,\epsilon)$. Let $f(s)=F_k(h_s)$. Then
\[
f'(s)=(n-k)\int_{\sph}{\phi\,S_{n-k-1}(\Xi^{-1}_s)d\mathcal{H}^{n-1}}\,,\quad
s\in(-\epsilon,\epsilon)\,.
\]
\end{prop}

\begin{proof}
\begin{eqnarray}
\label{IV.1}
f'(s)&=&\int_{\sph}{\frac{\p}{\p
    s}\,h_s\,S_{n-k-1}(\Xi^{-1}_s)d\mathcal{H}^{n-1}}\nonumber\\
&=&
\int_{\sph}\left[\phi\,S_{n-k-1}(\Xi^{-1}_s)+h_s\,\frac{\p}{\p
    s}(S_{n-k-1}(\Xi^{-1}_s))\right]d\mathcal{H}^{n-1}\\
&=&\int_{\sph}\left[\phi\,S_{n-k-1}(\Xi^{-1}_s)+
h_s\,\sum_{i,j=1}^{n-1}{S_{n-k-1}^{ij}(\Xi^{-1}_s)(\phi_{ij}+\phi\delta_{ij})}\right]d\mathcal{H}^{n-1}\,.\nonumber
\end{eqnarray}
Integrating by parts twice and using Lemma \ref{lemmaII.1} we obtain
\begin{equation}
\label{IV.2}
\int_{\sph}{h_s\sum_{i,j=1}^{n-1}{S_{n-k-1}^{ij}(\Xi_s^{-1})\phi_{ij}}d\mathcal{H}^{n-1}}=\int_{\sph}
{\phi\sum_{i,j=1}^{n-1}{S_{n-k-1}^{ij}(\Xi_s^{-1})(h_s)_{ij}d\mathcal{H}^{n-1}}}\,.
\end{equation}
On the other hand, by (\ref{I.1})
\begin{equation}
\label{IV.3}
\sum_{i,j=1}^{n-1}S_{n-k-1}^{ij}(\Xi_s^{-1})((h_s)_{ij}+h_s\delta_{ij})=(n-k-1)\,S_{n-k-1}(\Xi_s^{-1})\,.
\end{equation}
The proof is completed inserting (\ref{IV.2}) and (\ref{IV.3}) in (\ref{IV.1}).
\end{proof}

The proof of the next result a straightforward consequence of Proposition \ref{propIV.1}.

\begin{prop}\label{propIV.2}
In the assumptions and notations of Proposition \ref{propIV.1}
\beq
\label{f'' expression}
f''(0)=(n-k)\int_{\sph}{\phi
\sum_{i,j=1}^{n-1}{S_{n-k-1}^{ij}(\Xi^{-1})(\phi_{ij}+\phi \delta_{ij})}}
d\mathcal{H}^{n-1}\,.
\eeq
\end{prop}

We are now ready to prove Theorems \ref{teo0.1} and \ref{teo0.2}; we begin
with the latter.

\begin{proof}[Proof of Theorem \ref{teo0.2}.] Without loss of generality we
  may assume that $\phi\in C^{\infty}(\sph)$.  Fix $\epsilon>0$ such that $h+s\phi\in
\mathcal{C}$ for $s\in(-\epsilon,\epsilon)$ and let $k=n-J-1$.
As above, we set $f(s)=F_k(h+s\phi)$ and define
$g(s)=f^{\frac{1}{n-k}}(s)$. We know from Proposition \ref{BM for quermass} that $g$ is a concave
function and so
\begin{equation*}
%\label{IV.4}
g''(0)=
\frac{1}{n-k}\left[\left(\frac{1}{n-k}-1\right)f(0)^{\frac{1}{n-k}-2}(f'(0))^{2}+
(f(0))^{\frac{1}{n-k}-1}f''(0)\right]\leq0\,.
\end{equation*}
Notice that, by Proposition \ref{propII.1}, the assumption \eqref{0.1} gives
exactly $f'(0)=0$, so the condition $g''(0)\leq 0$ becomes
$(f(0))^{\frac{1}{I}-1}f''(0)\leq 0$. Since $f(0)=W_k(K)>0$ it
follows $f''(0)\leq 0.$ Now \eqref{f'' expression} gives us
\[
\int_{\sph}{\phi^2
\sum_{i,j=1}^{n-1}{S_{J}^{ij}(\Xi^{-1})\delta_{ij}}d\mathcal{H}^{n-1}}\leq
-\int_{\sph}{\phi\sum_{i,j=1}^{n-1}{S_{J}^{ij}(\Xi^{-1})\phi_{ij}}d\mathcal{H}^{n-1}}\,.
\]
Integrating by parts in the right hand--side and using Lemma
\ref{lemmaII.1} we obtain
\[
\int_{\sph}{\phi\sum_{i,j=1}^{n-1}{S_{J}^{ij}(\Xi^{-1})\phi_{ij}}d\mathcal{H}^{n-1}}=
-\int_{\sph}{\sum_{i,j=1}^{n-1}{S_{J}^{ij}(\Xi^{-1})\phi_i\phi_j}d\mathcal{H}^{n-1}}
\]
and we are done with the aid of part {\em v)} of Proposition \ref{propII.1}.
\end{proof}
For the proof of Theorem \ref{teo0.1} we need the following auxiliary result.

\begin{lema}\label{rhs} Let $\phi\in C^\infty(\sfe)$ and
  $\psi(x)=\phi(\nu(x))$, $x\in\partial K$, where $\nu$ is the Gauss map of
  $K$. Fix $r\in\{1,\dots,n-1\}$. Then for every $y\in\sfe$
\[
\frac{1}{\det({\Xi^{-1}(y)})}
\left<(S_{r}^{ij}(\Xi^{-1}(y)))\nabla\phi(y),\nabla\phi(y)\right>
=
\left<((D\nu(x))^{-1}(\nabla\psi(x)),S_{n-r}^{ij}(\Xi(x))\nabla\psi(x))\right>\,,
\]
where $x=\nu^{-1}(y)$ and $\Xi(x)=D\nu(x)$.
\end{lema}

\begin{proof}
We may assume that $\Xi^{-1}(y)$ is diagonal:
$$
\Xi^{-1}(y)={\rm diag}(\lambda_1,\dots,\lambda_{n-1})\,,\quad
\lambda_i>0\,,\;i=1,\dots,n-1\,.
$$
Then
$$
D\nu(x)={\rm diag}(\mu_1,\dots,\mu_{n-1})\,,\quad
\mu_i=\frac{1}{\lambda_i}\,,\;i=1,\dots,n-1\,.
$$
In particular
\begin{equation}\label{IV.4}
\nabla\psi(x)=D\nu(x)\nabla\phi(\nu(x))=\sum_{i=1}^{n-1}\mu_i\phi_i(y)\,.
\end{equation}
By Proposition \ref{propII.1} the matrix $(S_r^{ij}(\Xi^{-1}(y)))$ is also
diagonal and its eigenvalues are given by
$$
\Lambda_s=S_{r-1}({\rm diag}(\lambda_1,\dots,\lambda_{s-1},\lambda_{s+1},\dots,\lambda_{n-1}))\,,\quad
s=1,\dots,n-1\,.
$$
Using again Proposition \ref{propII.1} we get
\begin{eqnarray*}
\frac{\sum_{i,j=1}^{n-1}{S_{r}^{ij}(\Xi^{-1}(y))\phi_i(y)\phi_j(y)}}{\det({\Xi^{-1}(y)})}
&=&\sum_{i=1}^{n-1}{\frac{\Lambda_i}{\det(\Xi^{-1})}\,\phi_i^2(y)}\\
&=&\sum_{i=1}^{n-1}\mu_iS_{n-r-1}({\rm
  diag}(\mu_1,\dots,\mu_{i-1},\mu_{i+1},\dots,\mu_{n-1}))\,
\phi_i^2(y)\\
&=&\sum_{i=1}^{n-1}\mu_iS^{ii}_{n-r}(D\nu(x))
\phi_i^2(y)\\
&=&\langle\nabla\psi(x),(S^{ij}_{n-r}(D\nu(x)))\nabla\phi(y)\rangle\,.
\end{eqnarray*}
The conclusion of the lemma follows from the first equality in (\ref{IV.4})
and the symmetry of the matrix $(S^{ij}_{n-r}(D\nu(x)))$.
\end{proof}

\begin{proof}[Proof of Theorem \ref{teo0.1}.] We set $\phi(y)=\psi(\nu^{-1}(y))$,
  $y\in\sfe$. Consider the map $\nu^{-1}\,:\,\sfe\to\partial K$; its Jacobian is given by
$$
\det(D(\nu^{-1})(y))=\det(\Xi^{-1}(y))>0\,,\quad\forall y\in\sfe\,.
$$
Moreover, by Proposition \ref{propII.1} we have that for every $r\in\{0,1,\dots,n-1\}$,
$$
S_{r}(D\nu(\nu^{-1}(y)))=\frac{S_{n-r-1}(\Xi^{-1}(y))}{\det(\Xi^{-1}(y))}\,,\quad\forall
y\in\sfe\,.
$$
Hence we can write
\begin{eqnarray*}
\int_{\p K}{\psi S_{I-1}(D\nu)d\mathcal{H}^{n-1}}&=&
\int_{\sph}{\phi S_{n-I}(\Xi^{-1})d\mathcal{H}^{n-1}}\,,\\
\int_{\p K}{\psi^2 S_I(D\nu)d\mathcal{H}^{n-1}}&=&
\int_{\sph}{\phi^2 S_{n-I-1}(\Xi^{-1})d\mathcal{H}^{n-1}}\,.
\end{eqnarray*}
And, by Lemma \ref{rhs},
$$
\int_{\p
K}{\left<S_{I}^{ij}(D\nu)\nabla \psi,(D\nu)^{-1}\nabla \psi
\right>}d\mathcal{H}^{n-1}=
\int_{\sfe}\left<(S_{n-I}^{ij}(\Xi^{-1}))\nabla\phi,\nabla\phi\right>
d\mathcal{H}^{n-1}\,.
$$
The proof is completed applying Theorem \ref{teo0.2} with $J=n-I$.
\end{proof}

\begin{remark}\label{ossIII.1} {\em With the notation of the proof of Theorem
  \ref{teo0.2}, let $\phi(y)=\left<y_0,y\right>$, where
  $y_0\in\mathbb{S}^{n-1}$ is fixed. Note that condition (\ref{0.1}) is
  verified as
$$
\int_{\sph}{y S_{J}(h_{ij}(y)+h\delta_{ij}(y))
d\mathcal{H}^{n-1}}= \int_{\sph}y\,dA_J(K,y)\,,
$$
where $A_J(K,\cdot)$ is the $J$--th area measure of $K$ (see
\cite{Schneider} or the next section for the definition), and the
latter integral is zero by standard properties of area measures.
Moreover, for every $s$, $h+s\phi$ is the support function of a
translate of $K$ . Since quermassintegrals are invariant with
respect to translations, the function $f$ is constant in 
particular $f''(0)=0$. This proves that if $\phi$ is as above we
have equality in (\ref{0.1b}). Analogously, choosing
$\psi(x)=\left<x_0,\nu(x)\right>$ where $0\ne x_0\in\R^n$ is
fixed, we see that condition (\ref{0.02}) of Theorem \ref{teo0.1}
is fulfilled and (\ref{0.03}) becomes an equality.}
\end{remark}

\section{The case $J=1$: the proof of Theorems \ref{teo0.3} and \ref{teo0.4}}\label{V}
\label{IV}
We start this section recalling the definition of area measures; for a
detailed presentation of this topic we refer the reader to \cite[Chapter 5]{Schneider}. If
$K_1,\dots,K_m$, $m\in\mathbb{N}$, are convex bodies in $\R^n$ and
$\lambda_1,\dots,\lambda_m$ are non--negative real numbers, then we have:
$$
{\cal
  H}^n(\lambda_1K_1+\dots+\lambda_mK_m)=\sum_{i_1,\dots,i_n=1}^m\,\lambda_{i_1}\cdots\lambda_{i_n}
V(K_{i_1},\dots,K_{i_n})\,.
$$
The coefficients of the polynomial at the right hand--side are called {\it
  mixed volumes}. Moreover, if we fix $(n-1)$ convex bodies
$K_2,\dots,K_{n}$, there exists a unique non--negative Borel measure
$A(K_2,\dots,K_n,\cdot)$ (called {\it mixed area measure}) such that for every
convex body $K_1$
$$
V(K_1,K_2,\dots,K_n)=\int_{\mathbb{S}^{n-1}}h_{K_1}(x)\,dA(K_2,\dots,K_n,x)\,.
$$
For $j=1,\dots,n-1$, the {\it area measure} of order $j$ of a convex body $K$
is obtained in the following way: $A_j(K,\cdot)=A(K,\dots,K,B,\dots,B,\cdot)$, where $K$
is repeated $j$ times and $B$ is the unit ball in $\R^n$. An alternative definition of area
measures is based on a local version of the Steiner formula (see \cite[Chapter
4]{Schneider}). In particular, the 
area measure of order one of $K$ is $A_1(K,\cdot)=A(K,B\dots,B,\cdot)$. If $K$ is of class
$C^2_+$, then it can be proved that
\begin{equation}
\label{IV.0}
dA_1(K,x)=\frac{1}{n-1}S_1((h_K)_{ij}(x)+h_K(x)\delta_{ij})d{\cal
H}^{n-1}(x)\,.
\end{equation}
Hence condition
(\ref{0.1}) is equivalent to (\ref{0.2}) when $h$ is the support function of a
convex body of class $C^2_+$.

\begin{proof}[Proof of Theorem \ref{teo0.3}] We may assume that $\phi\in
  C^\infty(\mathbb{S}^{n-1})$. $K$ can be
approximated by a sequence $K_r$, $r\in\mathbb{N}$, such that for
every $r,$ $K_r$ is of class $C^2_+$ and $(K_r)_{r\in\N}$
converges to $K$ in the Hausdorff metric as $r$ tends to infinity.
Fix $r\in\mathbb{N}$ and let $h_r$ be the
support function of $K_r$. For $s$ sufficiently small in absolute
value, consider the function
$$
f_r(s)=\int_{\sfe}(h_r+s\phi)S_1((h_r+s\phi)_{ij}+(h_r+s\phi)\delta_{ij})\,d{\cal H}^{n-1}\,.
$$
By Proposition \ref{BM for quermass}, $\sqrt{f_r}$ is concave so
that $2f_r(0)f_r''(0)-(f_r'(0))^2\le0$. Using (\ref{I.0}),
Propositions \ref{propIV.1} and \ref{propIV.2} (with $k=n-2$) and
the relation $(S^{ij}_1)=(\delta_{ij})$, we obtain
\begin{equation}\label{V.1}
\frac{2n}{n-2}W_{n-2}(K_r)\int_{\sfe}\phi\left((n-1)\phi+\sum_{i=1}^{n-1}\phi_{ii}\right)\,d{\cal H}^{n-1}\le
\left(\int_{\sfe}\phi S_1((h_r)_{ij}+h_r\delta_{ij})\,d{\cal H}^{n-1}\right)^2\,.
\end{equation}
From (\ref{IV.0}) we know that
$$
\int_{\sfe}\phi S_1((h_r)_{ij}+h_r\delta_{ij})\,d{\cal H}^{n-1}=
(n-1)\int_{\sfe}\phi(x)\,dA_1(K_r,x)\,,
$$
where $A_1(K_r,\cdot)$ is the first area measure of $K_r$.
Moreover, as $r$ tends to infinity the sequence of measures
$A_1(K_r,\cdot)$ converges weakly to $A_1(K,\cdot)$ (see
\cite[Theorem 4.2.1]{Schneider}). This implies
\begin{equation}
\label{IV.1b}
\lim_{r\to\infty}\int_{\sfe}\phi(x)\,dA_1(K_r,x)=
\int_{\sfe}\phi(x)\,dA_1(K,x)=0\,.
\end{equation}
On the other hand $W_{n-2}(K_r)$ converges to $W_{n-2}(K)$ as $r$ tends to
infinity (by standard continuity results on quermassintegrals) and
$W_{n-2}(K)>0$ as $K$ has interior points. The conclusion follows letting
$r\to\infty$ in (\ref{V.1}), using (\ref{IV.1b}) and integrating by parts.
\end{proof}

As mentioned in the Introduction, Theorem \ref{teo0.3} extends the usual
(sharp) Poincar\'e inequality (\ref{0.1c}) on $\sfe$ when the usual
zero--mean condition is replaced by (\ref{0.2}). Clearly, in order to apply
this result it would be useful
to understand when a measure $\mu$ on $\sfe$ is the area measure
of order one of some convex body. This amounts to solve the
Christoffel problem for $\mu$ (see for instance \cite[\S
4.3]{Schneider}). For $n=2$ this problem coincides with the
Minkowski problem and its solution is completely understood. Let
$\mu$ be a non--negative Borel measure on $\mathbb{S}^1$ such that:
{\it i)} $\mu$ is not the sum of two point--masses; {\it ii)}
$$
\int_{\mathbb{S}^1}x\,d\mu(x)=0\,.
$$
Then there exists a convex body $K$ in $\R^2$ such that
$A_1(K,\cdot)=\mu(\cdot)$ (note that conditions {\it i)} and {\it ii)} are
also necessary in order that $\mu$ is the area measure of order one of some convex
body). Hence we have the following extension of the well known
{\it Wirtinger inequality}.

\begin{cor}\label{cor dim2} Let $\nu$ be a non--negative Borel measure on $[0,2\pi]$ such that
$\nu$ is not the sum of two point--masses and
$$
\int_0^{2\pi}\sin\theta\,d\nu(\theta)=\int_0^{2\pi}\cos\theta\,d\nu(\theta)=0\,.
$$
Then, for every $\phi\in C^1([0,2\pi])$ such that $\phi(0)=\phi(2\pi)$
$$
\int_0^{2\pi}\phi(\theta)\,d\nu(\theta)=0\
\quad\Rightarrow\quad
\int_0^{2\pi}(\phi(\theta))^2\,d\theta\le\int_0^{2\pi}(\phi'(\theta))^2\,d\theta\,.
$$
\end{cor}

In higher dimension the Christoffel problem is more complicated.
Necessary and sufficient conditions for a measure $\mu$ to be the
first area measure of some convex body were found by Firey
\cite{Firey} and Berg \cite{Berg} (see also \cite[\S
4.2]{Schneider}). On the other hand these conditions are not easy
to use in practice. A considerable progress (in a larger class of
problems) has been made by Guan and Ma in \cite{Guan-Ma} and Sheng,
Trudinger and Wang in \cite{Sheng-Trudinger-Wang} where a rather
simple sufficient condition is found. Here we state this result in the case of
area measures of order one.

\begin{theorem}[{\bf Guan, Ma, Sheng, Trudinger, Wang}]\label{teoV.1} Let $f\in
  C^{1,1}(\sfe)$, $f>0$ and let $g=1/f$. If
$$
\int_{\sfe} x f(x)d\mathcal{H}^{n-1}(x)=0\,,
$$
and the matrix $(g_{ij}+g\delta_{ij})$ is positive semi--definite a.e. on
$\sfe$, then there exists a convex body $L$, uniquely determined up to
translations, such that
$$
dA_1(L,\cdot)=f(\cdot)\,d\mathcal{H}^{n-1}(\cdot)\,,
$$
i.e. $f$ is the density of $S_1(K,\cdot)$ with respect to
$\mathcal{H}^{n-1}(\cdot)$.
\end{theorem}

Using the above result and Theorem \ref{teo0.3}, we now proceed to show Theorem
\ref{teo0.4}.

\begin{proof}[Proof of Theorem \ref{teo0.4}.]
We recall that the radial function $\rho_K$ of $K$ is defined as
$\rho_K(x)=\max\{\lambda\ge0\,|\,\lambda x\in K\}$. Let $H$ be the polar body
of $K$:
$$
H=\{x\in\mathbb{R}^n\,:\,\left<x,y\right>\le1\,,\;\forall y\in K\}\,.
$$
$H$ is still a convex body and the origin belongs to its interior. Note that (see for
instance \cite[Remark 1.7.7]{Schneider})
$$
\rho_K=\frac{1}{h_{H}}\,,\quad\mbox{on $\sfe$.}
$$
Let $H_r$, $r\in\mathbb{N}$, be a sequence of convex bodies
converging to $H$ in the Hausdorff metric as $r$ tends to
infinity, such that each $H_r$ is of class $C^2_+$. By hypothesis
(\ref{0.4}) we may assume that
$$
\int_{\sfe}x\,\frac{1}{h_{H_r(x)}}d\mathcal{H}^{n-1}(x)=0\quad\forall
r\in\mathbb N\,.
$$
Setting $h_r=h_{H_r}$ we have that $h_r\to h_H$
uniformly on $\sfe$ and
\begin{equation}\label{V.2}
((h_r)_{ij}+h_r\delta_{ij})>0\quad\mbox{on $\sfe$ for every $r\in\N$.}
\end{equation}
Hence for every $r\in\mathbb{N}$ we can apply Theorem \ref{teoV.1} with
$f=f_r=1/h_r$, obtaining a convex body $L_r$ such that
$$
dA_1(L_r,\cdot)=f_r(\cdot)\,d\mathcal{H}^{n-1}(\cdot)\,.
$$
As $H$ is a convex body with interior points, we have that
$c<h_H<C$ on $\sfe$, for suitable positive constants $c$ and $C$.
Using the uniform convergence we obtain that there exist $d,D>0$
such that $d\le f_r(x)\le D$, $\forall\,x\in\sfe$, $\forall
r\in\mathbb{N}$. Hence we may apply Lemma 3.1 in \cite{Guan-Ma} to
deduce that the sequence $L_r$ is bounded and by the Blaschke
selection theorem (see \cite[Theorem 1.8.6]{Schneider}), up to a
subsequence, it converges to a convex body $L$ in the Hausdorff
metric. As already noticed in the proof of Theorem \ref{teo0.3},
the sequence of measures $A_1(L_r,\cdot)$ converges weakly to
$A_1(L,\cdot)$ as $r$ tends to infinity. Consequently
$$
dA_1(L,\cdot)=\frac{1}{h_H(\cdot)}\,d\mathcal{H}^{n-1}(\cdot)=\rho_K(\cdot)d\mathcal{H}^{n-1}(\cdot)\,.
$$
The conclusion follows applying Theorem \ref{teo0.3}.
\end{proof}

\bigskip

\noindent
{\sc Andrea Colesanti}\\
Dipartimento di Matematica 'U. Dini' -- Universit\`a di Firenze\\
Viale Morgagni 67/a\\
50134 Firenze, Italy\\
{\tt colesant@math.unifi.it}

\medskip

\noindent
{\sc Eugenia Saor\'{\i}n G\'{o}mez}\\
Departamento de Matematicas -- Universidad de Murcia \\
Campus de Espinardo\\
30100 Murcia, Spain\\
{\tt esaorin@um.es}

\end{document}